\documentclass[11pt,a4paper]{amsart}
\usepackage{amssymb}
\usepackage{mathabx}
\usepackage{mathrsfs}
\usepackage{syntonly}
\usepackage{amsmath}
\usepackage{amsthm}
\usepackage{amsfonts}
\usepackage{amssymb}
\usepackage{latexsym}
\usepackage{amscd,amssymb,amsopn,amsmath,amsthm,graphics,amsfonts,mathrsfs,accents,enumerate,verbatim,calc}
\usepackage[dvips]{graphicx}
\usepackage[colorlinks=true,linkcolor=red,citecolor=blue]{hyperref}
\usepackage[all]{xy}

\date{}
\allowdisplaybreaks[4] \footskip=15pt
\renewcommand{\uppercasenonmath}[1]{}

\numberwithin{equation}{section} \theoremstyle{plain}
\newtheorem{lem}{Lemma}[section]
\newtheorem{cor}[lem]{Corollary}
\newtheorem{prop}[lem]{Proposition}
\newtheorem{thm}[lem]{Theorem}
\newtheorem{hyp}[lem]{Hypothesis}

\newtheorem{nota}[lem]{Notation}

\newtheorem{definition}[lem]{Definition}
\newtheorem{Ex}[lem]{Example}
\newtheorem{Quest}[lem]{Question}
\newtheorem{Property}[lem]{Property}
\newtheorem{Properties}[lem]{Properties}
\newtheorem{Subprops}{}[lem]
\newtheorem{Para}[lem]{}

\newtheorem{rem}[lem]{Remark}

\newenvironment{df}{\begin{definition}\rm}{\end{definition}}
\newenvironment{ex}{\begin{Ex}\rm}{\end{Ex}}

\newtheorem*{ack*}{ACKNOWLEDGEMENTS}

\newcommand{\pf}{\noindent\begin {proof}}
\newcommand{\epf}{\end{proof}}

\begin{document}
\begin{center}
{\large  \bf  Tensor Abelian Geometry of  VI-modules}

\vspace{0.5cm}  Peng Xu$\footnote{Corresponding author. Peng Xu is supported by the National Natural Science Foundation of China (Grant No. 12231009, 11971224) }$

\end{center}

\bigskip
\centerline { \bf  Abstract}
\medskip

\leftskip10truemm \rightskip10truemm \noindent\hspace{1em}  In this short note,~we study the spectrum of prime Serre ideals of global representations for noetherian families. In particular, we prove that the spectrum of prime Serre ideals of finitely generated VI-modules is homeomorphic to $\mathbb{N}^{*}$, the one-point compactification of $\mathbb{N}$, which differs from the Balmer spectrum of derived VI-modules. Our method could also be applied to the category of finitely generated FI-modules and the category of global representations for the family of cyclic p-groups. \\[2mm]

\leftskip0truemm \rightskip0truemm
%\bigskip
\section { \bf Introduction}

Tensor abelian geometry, introduced by Buan, Krause and Solberg \cite{BKS}, associates to any essentially small tensor abelian category $(\mathcal{A},\otimes, \mathbf{1})$ a topological space $\mathrm{Spc}(\mathcal{A})$,  whose points are the prime Serre ideals in $\mathcal{A}$, endowed with the Zariski topology. This is an abelian analogue of tensor triangular geometry \cite{Bal05,BG22, BG25, Gal18, Gal19}. Our main goal in this short note is to study the spectrum of the abelian category of finitely generated VI-modules, in the setting of global representations for a family of finite groups.

Fix a field $k$ of characteristic zero. Let $\mathcal{G}$  be the category of finite groups and conjugacy classes of surjective group homomorphisms. For  a replete, full subcategory $\mathcal{U}\subseteq \mathcal{G}$, the category $\mathrm{A}(\mathcal{U})$ of $\mathcal{U}$-global representations is defined as the Grothendieck category of all contravariant functors from $\mathcal{U}$ to $k$-vector spaces. This concept was introduced by Pol and Strickland and is closely related to categories of VI-modules and FI-modules; see \cite{PS22}. In particular, when $\mathcal{U}$ is taken to be the category of elementary abelian p-groups, Pontryagin duality induces an equivalence between the category of VI-modules (over $k$) and $\mathrm{A}(\mathcal{U})$.  We say $\mathcal{U}$ is  noetherian if $\mathrm{A}(\mathcal{U})$ is a locally noetherian Grothendieck category. Important examples include the category of elementary abelian p-groups and the category of cyclic p-groups. In the noetherian case,  the subcategory $\mathrm{A}(\mathcal{U})^{c}$ of finitely generated objects forms a tensor abelian category. Here we study the tensor abelian geometry of $\mathrm{A}(\mathcal{U})^{c}$  using methods adapted from \cite{BBP+25b}.

For an object $X\in \mathrm{A}(\mathcal{U})^{c}$, its support is defined as the subcategory of $\mathcal{U}$ on which $X$ takes nonzero values. The pair $(\pi_{0}\mathcal{U}, \mathrm{supp})$ constitutes a support datum on $\mathrm{A}(\mathcal{U})^{c}$, and consequently gives rise to a distinguished class of   prime Serre ideals in $\mathrm{A}(\mathcal{U})^{c}$, called group primes. Moreover, this support serves as a key invariant for understanding  relations among objects in $\mathrm{A}(\mathcal{U})$. Our main theorem in this direction is stated below. (For the definition of $\mathrm{Serre}_{\otimes}^{+}\langle Y\rangle$, see Definition \ref{df 4.13}.)

\vspace{2mm}
 {\bf Theorem 1.1 }  Let $\mathcal{U}$ be noetherian and $X, Y\in \mathrm{A}(\mathcal{U})^{c}$. If $\mathrm{supp}(X)\subseteq \mathrm{supp}(Y)$, then $X\in \mathrm{Serre}_{\otimes}^{+}\langle Y\rangle$.\\

Motivated by structural  properties of finitely generated VI-modules and FI-modules, we introduce the notion of an $\mathbb{N}$-stable family (see Definition \ref{df 5.1}). When $\mathcal{U}$ is $\mathbb{N}$-stable and noetherian, there exists a new prime Serre ideal $P_{\infty}$ in $\mathrm{A}(\mathcal{U})^{c}$, which turns out to be the unique prime Serre ideal that is not a group prime. Specifically, we prove the following result:

 \vspace{2mm}
  {\bf Theorem 1.2 }  Let $\mathcal{U}$ be  noetherian and $\mathbb{N}$-stable. Then the spectrum of prime Serre ideals in  $\mathrm{A}(\mathcal{U})^{c}$ is homeomorphic to $\mathbb{N}^{*}$, the one-point compactification of $\mathbb{N}$.\\

In particular, the spectrum of $\mathrm{A}(\mathcal{U})^{c}$ is homeomorphic to $\mathbb{N}^{*}$ when $\mathcal{U}$ is the category of elementary abelian p-groups or the category of cyclic p-groups. We mention that this result  provides a nontrivial example of tensor abelian category for which the spectrum of prime Serre ideals can be explicitly computed and it illustrates that a noetherian tensor abelian category may admit a non-noetherian spectrum. Notably,  the spectrum  $\mathrm{Spc}(\mathrm{A}(\mathcal{U})^{c})$ does not depend on whether $\mathcal{U}$ is multiplicative---in sharp contrast to the Balmer spectrum of $\mathrm{D}(\mathcal{U})^{c}$, the subcategory of compact objects in the derived category of $\mathrm{A}(\mathcal{U})$, where multiplicativity plays a crucial role, see \cite[Theorem 6.13, Example 12.10]{BBP+25b}. It should also be noted that our method naturally extends to the category of finitely generated FI-modules.\\

\textbf{Conventions}: Throughout, we  will work over a field $k$ of characteristic 0. We write $G\twoheadrightarrow H$ to denote a surjective group homomorphism. All categories considered in this paper are assumed to be samll. By a subcategory, we always mean one that is replete and full.

\section { \bf Preliminaries}
In this section, we briefly recall some of the notation and constructions which we will use.  For more details, we refer the reader to \cite{PS22,BBP+25a}.

Throughout this paper,  let $\mathcal{G}$ be the category of finite groups and conjugacy classes of surjective group homomorphisms. A replete full subcategory $\mathcal{U}$ of $\mathcal{G}$ is said closed downwards if for any surjective group homomorphism $G\twoheadrightarrow H$, $G\in \mathcal{U}$ implies $H\in \mathcal{U}$. $\mathcal{U}$  is said closed upwards if for any surjective group homomorphism $H\twoheadrightarrow G$, $G\in \mathcal{U}$ implies $H\in \mathcal{U}$. $\mathcal{U}$  is said widely closed  if whenever $G\twoheadleftarrow H\twoheadrightarrow K$ are surjective homomorphisms with $G,H,K\in \mathcal{U}$, the image of the combined morphism $H\rightarrow G\times K$ is also in $\mathcal{U}$. In particular, if $\mathcal{U}$ is closed downwards or closed upwards, then $\mathcal{U}$ is widely closed. $\mathcal{U}$ is said essentially finite if it contains only finite many isomorphism classes of objects. We write $\pi_{0}(\mathcal{U})$ for the set of isomorphism classes of objects of $\mathcal{U}$. We denote by $\mathcal{U}_{> n}$ the full subcategory of $\mathcal{U}$ consisting of groups of order greater than
$n$, and by $\mathcal{U}_{\leq n}$ the full subcategory consisting of groups of order at most $n$.  The upwards closure of a subset $S\subset \pi_{0}(\mathcal{U})$ is $$\uparrow(S):=\{~[G]\in \pi_{0}(\mathcal{U}) |~\exists [H]\in S:~G\twoheadrightarrow H\}.$$

 Given a replete full subcategory $\mathcal{U}\subseteq \mathcal{G}$, we denote by $$\mathrm{A(\mathcal{U}):=Fun(\mathcal{U}^{op};Mod}~k)$$ the abelian category of functors $\mathrm{\mathcal{U}^{op}\rightarrow Mod}~k.$ Let $X\in \mathrm{A(\mathcal{U})}$, the (index) support of $X$ is defined to be $\mathrm{supp}(X):=\{[G]\in\pi_{0}(\mathcal{U})|X(G)\neq 0\}$. When $\{G\}$ is the replete full subcategory spanned by a single finite group $G$, we have $\mathrm{A}(\{G\})\simeq \mathrm{Mod}~k[\mathrm{Out}(G)]$.

Let $\mathcal{U}$ be a subcategory and $G\in \mathcal{U}$. There is an evaluation functor $$ev_{G}:\mathrm{A}(\mathcal{U})\rightarrow \mathrm{A}(\{G\})\quad\quad\quad X\mapsto X(G)$$ which admits a left adjoint $$e_{G,\bullet}:\mathrm{A}(\{G\})\rightarrow \mathrm{A}(\mathcal{U})\quad\quad\quad V\mapsto e_{G,V}\cong V\otimes_{k[\mathrm{Out}(G)]}k[\mathrm{Hom}_{\mathcal{U}}(-,G)].$$ Set $e_{G}:=e_{G,k[\mathrm{Out}(G)]}$, then we have $\mathrm{Hom}_{\mathrm{A}(\mathcal{U})}(e_{G},X)\cong X(G)$. Let V be an $\mathrm{Out}(G)$-representation, the object $\chi_{G,V}$ is defined as following: $\chi_{G,V}(H)=e_{G,V}(H)$ if $G\cong H$ and  $\chi_{G,V}(H)=0$ otherwise. There is a map $e_{G,V}\rightarrow \chi_{G,V}$ induced by the identity map of $V$.

We have the following facts about $\mathrm{A}(\mathcal{U})$ .
\begin{enumerate}
\item $\mathrm{A}(\mathcal{U})$ is a Grothendieck abelian category with generators given by $e_{G}$ for each $G\in \mathcal{U}$.

\item $\mathrm{A}(\mathcal{U})$ is a symmetric monoidal category with pointwise tensor product,i.e.,$(X\otimes Y)(G)=X(G)\otimes Y(G)$ for $X,Y\in \mathrm{A}(\mathcal{U})$ and $G\in \mathcal{U}$. The tensor unit $\mathrm{1}$ is the constant functor with value $k$ and all maps the identity. When $\mathcal{U}$ contains the trivial group, then $\mathrm{1}\cong e_{1}$.

\item Each object in $\mathrm{A}(\mathcal{U})$ is flat.

\end{enumerate}

Given a functor $f:\mathcal{U}\rightarrow \mathcal{V}$, we have adjunctions
\[
\xymatrix@R=0em@C=4em{
\mathsf{A}(\mathcal{U})
\ar@/^3ex/[r]^{f_{!}}        % 上弧线向右
\ar@/_3ex/[r]_{f_{*}}        % 下弧线向右
&
\mathsf{A}(\mathcal{V})
\ar[l]^{f^{*}}               % 从A(V)向左画箭头到A(U)
}
\]
where $f^{*}$ is defined by $f^{*}(X)(H)=X(f(H))$, the left adjoint $f_{!}$ of $f^{*}$ is given by the left Kan extension along $f$, and the right adjoint $f_{*}$ of $f^{*}$ is given by the right Kan extension along $f$.

Recall that an object $X\in \mathrm{A}(\mathcal{U})$ is finitely generated if there exists an epimorphism $\oplus^{n}_{i=1}e_{G_{i}}\twoheadrightarrow X.$

The following facts will be used  freely.

Let $i:\mathcal{U}\rightarrow \mathcal{V}$ be an inclusion of replete full subcategories of $\mathcal{G}$. Then:
\begin{enumerate}
\item $i^{*}i_{*}(X)\cong X\cong i^{*}i_{!}(X)$ for $X\in \mathrm{A}(\mathcal{U})$.

\item $i^{*}(e_{G,V})=e_{G,V}$ and $i_{!}(e_{G,V})=e_{G,V}$.

\item $i^{*}$  is strong monoidal, i.e., $i^{*}(\mathrm{1})=\mathrm{1}$ and $i^{*}(X\otimes Y)=i^{*}(X)\otimes i^{*}(Y)$.

\item There are natural maps $i_{!}(\mathrm{1})\rightarrow \mathrm{1}\rightarrow i_{*}(\mathrm{1})$ and $i_{!}(X\otimes Y)\rightarrow i_{!}(X)\otimes i_{!}(Y)$ and $i_{*}(X)\otimes i_{*}(Y)\rightarrow i_{*}(X\otimes Y)$ giving (op)lax monoidal structures.

\item If $\mathcal{U}$  is closed upwards in $\mathcal{V}$, then $i_{!}$ is extension by zero and so preserves  all limits, colimits and tensors (but not the unit).

\item If $\mathcal{U}$  is closed downwards in $\mathcal{V}$, then $i_{*}$ is extension by zero and so preserves all limits, colimits and tensors (but not the unit).

\item If $\mathcal{U}$  is closed downwards in $\mathcal{V}$, then $i^{*}$  preserves finitely generated objects.

\item $i_{!}$  preserves finitely generated objects.
\end{enumerate}
We will always assume that $\mathcal{U}$  is widely closed and the unit $\mathrm{1}$ is finitely generated.

 \section{\bf Tensor abelian geometry}
In this section, we study the spectrum of prime serre ideals of tensor abelian categories. This is an analogue of Balmer's tensor triangular geometry \cite{Bal05}. The lattice-theoretic approach is due to Buan, Krause and  Solberg \cite{BKS}. By a tensor abelian category $\mathcal{A}$, we mean an abelian category equipped with a symmetric monoidal structure for which the tensor product is exact in each variable. A full subcategory $\mathcal{D}$ of $\mathcal{A}$ is called Serre if for any short exact sequence $0\rightarrow A\rightarrow B \rightarrow C\rightarrow 0$ in $\mathcal{A}$, we have $B\in \mathcal{D}$ if and only if $A,C\in \mathcal{D}$. A Serre subcategory $\mathcal{D}$ is called a Serre ideal if $X\otimes Y\in \mathcal{D}$ for any $X\in \mathcal{D}$ and $Y\in \mathcal{A}$.  For a collection of objects $S$ of $\mathcal{A}$, we denote by $\mathrm{Serre}_{\otimes}\langle S\rangle$ the Serre ideal generated by $S$ and $\mathrm{Serre}\langle S\rangle$ the Serre subcategory generated by $S$.

\begin{df} A Serre ideal $\mathcal{D}$ of $\mathcal{A}$ is called radical if $X^{\otimes n}\in \mathcal{D}$ implies $X\in \mathcal{D}$ for any positive integer $n$ and $X\in \mathcal{A}$. A proper Serre ideal $\mathcal{P}$ of $\mathcal{A}$ is called prime if $X\otimes Y\in \mathcal{P}$ implies $X\in \mathcal{P}$ or $Y\in \mathcal{P}$ for $X,Y\in \mathcal{A}$.
\end{df}

\begin{rem}Let $L_{Serre}(\mathcal{A})$ be the set of Serre ideals of $\mathcal{A}$. It is an ideal lattice in the sense of \cite{BKS}.  A Serre ideal $\mathcal{D}$ is radical if and only if it is a semi-prime element in $L_{Serre}(\mathcal{A})$ and a proper Serre ideal $\mathcal{P}$ is prime if and only if it is a prime element in $L_{Serre}(\mathcal{A})$.
\end{rem}

\begin{df}Let $\mathrm{Spc}(\mathcal{A})$ be the set of all prime Serre ideals of $\mathcal{A}$. The Zariski topology on $\mathrm{Spc}(\mathcal{A})$ is defined as following: a subset of $\mathrm{Spc}(\mathcal{A})$ is closed if it is of the form $Z(S):=\{P\in \mathrm{Spc}(\mathcal{A})|~P\cap S=\varnothing \}$ for some family of objects $S\subset \mathcal{A}$.
\end{df}

\begin{rem} The definition of spectrum of prime Serre ideals is equivalent to that one defined in \cite[Section 7]{BKS}. We choose this  because it is convenient to work with support data defined on objects of $\mathcal{A}$ instead of support data defined on ideals. On the other hand, there exists another spectrum $\mathrm{Spec}(\mathcal{A})$ of $\mathcal{A}$, namely the spectrum of prime thick ideals; see \cite[Section 7]{Kra24} and \cite[Theorem 6.8]{B26}.
\end{rem}

\begin{ex}Let $Vec_{K}$ be the category of finite dimensional vector spaces over a field $K$. Then $\mathrm{Spc}(Vec_{K})=\{0\}.$
\end{ex}

\begin{ex}Let $K$ be a field of characteristic $p$ and $C_{p}$ the cyclic group of order $p$. Then the spectrum $\mathrm{Spc}(\mathrm{mod}~KC_{p})$ of prime Serre ideals of the category of finite dimensional $KC_{p}$-modules is $\{0\}$. Indeed, there exists two prime thick ideals of $\mathrm{mod}~KC_{p}$ by \cite[Example 17]{Kra24},  $\{0\}$ and  the subcategory of finite dimensional projective $KC_{p}$-modules. Since the latter is not closed under subobjects, the only prime Serre ideal is $\{0\}$.
\end{ex}

Let $\mathcal{D}$ be a Serre ideal of the tensor abelian category $\mathcal{A}$. Then the quotient category $\mathcal{A}/\mathcal{D}$ has the same objects as $\mathcal{A}$ and that its morphisms are obtained via calculus of fractions by inverting those morphisms having their kernel and cokernel in $\mathcal{D}$.  The quotient category $\mathcal{A}/\mathcal{D}$ inherits a symmetric monoidal structure since $\mathcal{D}$ is a Serre ideal. The following result shows that the spectrum of prime Serre ideals behaves well under such a localization. Compare with \cite[Proposition 3.11]{Bal05} and \cite[Lemma 6.10]{B26}.
\begin{prop}\label{prop3.7}Let $\mathcal{D}$ be a Serre ideal of the tensor abelian category $\mathcal{A}$ and $\pi: \mathcal{A}\rightarrow \mathcal{A}/\mathcal{D}$  the localization. Then there is a homeomorphism between $\mathrm{Spc}(\mathcal{A}/\mathcal{D})$ and the subspace $\{P\in \mathrm{Spc}(\mathcal{A})| \mathcal{D}\subset P\}$ of $\mathrm{Spc}(\mathcal{A})$.
\end{prop}
\begin{proof} Since $\pi$ is a tensor exact functor, there is a continuous map $\mathrm{Spc}(\pi): \mathrm{Spc}(\mathcal{A}/\mathcal{D})\rightarrow \mathrm{Spc}(\mathcal{A})$ by \cite[Lemma 8.2]{BKS}, which sends prime $Q\in \mathrm{Spc}(\mathcal{A}/\mathcal{D})$ to $\pi^{-1}(Q)\in \mathrm{Spc}(\mathcal{A})$. This is injective because we have $Serre_{\otimes}(\pi(\pi^{-1}(Q)))=Q$ for any $Q\in \mathrm{Spc}(\mathcal{A}/\mathcal{D})$.

Now let $V:=\{P\in \mathrm{Spc}(\mathcal{A})| \mathcal{D}\subset P\}$. The containment $\mathcal{D}=\pi^{-1}(0)\subset \pi^{-1}(Q)=\mathrm{Spc}(\pi)(Q)$ implies $\mathrm{Im}(\mathrm{Spc}(\pi))\subset V$. On the other hand we have $\pi^{-1}(\pi(P))=P$ for any $P\in V$, thus to show $\mathrm{Im}(\mathrm{Spc}(\pi))=V$, it suffices to prove $\pi(P)$ is a prime Serre ideal in $\mathcal{A}/\mathcal{D}$. Indeed, $\pi(P)$ is a tensor ideal since $\pi$ is essentially surjective.  If $\pi(X)\otimes \pi(Y)\in \pi(P)$, then $\pi(X\otimes Y)\cong \pi(Z)$ in $\mathcal{A}/\mathcal{D}$ for some $Z\in P$. This isomorphism can be represented by a roof diagram $$\xymatrix{
                & C \ar[dr]^{g}\ar[dl]_{f}             \\
 X\otimes Y  & &     Z       }
 $$
where $\mathrm{Ker}(f),\mathrm{Ker}(g),\mathrm{Coker}(f),\mathrm{Coker}(g)\in \mathcal{D}\subset P$. Then $X\otimes Y\in P$, which implies $X\in P$ or $Y\in P$. Thus $\pi(P)$ is prime. By a similar argument, one could verify that $\pi(P)$ is closed under subobjects, quotients and extensions. Thus there is a continuous bijection between $V$ and $\mathrm{Spc}(\mathcal{A}/\mathcal{D})$.

Finally, let $X\in \mathcal{A}, P\in V$. Then $\pi(X)\in \pi(P)$ if and only if $X\in P$ and so $\mathrm{Spc}(\pi)(Z(\pi(X)))=Z(X)\cap V$, which proves that $\mathrm{Spc}(\pi): \mathrm{Spc}(\mathcal{A}/\mathcal{D})\rightarrow V$ is closed. Hence the result.
\end{proof}

Recall that a support datum~\cite[Definition 7.9 ]{BKS} on $\mathcal{A}$ is a pair $(\mathcal{T}, \tau)$ where $\mathcal{T}$ is a topological space and $\tau$ is a map which assigns to each object $X\in \mathcal{A}$ a closed subset $\tau(X)\subseteq \mathcal{T}$, such that for all $X,Y\in \mathcal{A}$

$\tau(\mathrm{1})=\mathcal{T}$,~$\tau(X)=\bigcup_{Z\in \mathrm{Serre}_{\otimes}\langle X\rangle}\tau(Z)$,~$\tau(X\oplus Y)=\tau(X)\cup \tau(Y)$,~$\tau(X\otimes Y)=\tau(X)\cap \tau(Y)$.

A support datum $(\mathcal{T}, \tau)$ is called classifying if $\mathcal{T}$ is spectral and the assignments
$$ \mathcal{D}\mapsto \bigcup_{X\in \mathcal{D}}\tau(X) \quad\quad \mathcal{J}\mapsto \mathcal{A}_{\mathcal{J}}:=\{X\in \mathcal{A}|~\tau(X)\subseteq \mathcal{J}\}$$
induce bijections between
\begin{enumerate}
\item the set of radical Serre ideals of $\mathcal{A}$, and

\item the set of all subsets $\mathcal{J}\subseteq \mathcal{T}$ of the form $\mathcal{J}=\bigcup_{i\in \Omega}\mathcal{J}_{i}$ with quasi-compact open
complement $\mathcal{T}\setminus \mathcal{J}_{i}$ for each $i\in \Omega$.

\end{enumerate}

We end this section with a criterion for determining $\mathrm{Spc}(\mathcal{A})$, which is a translation of \cite[Corollary 6.2]{BKS}.
\begin{lem}\label{lem 3.8} If $(\mathcal{T}, \tau)$ is a classifying support datum on $\mathcal{A}$, then $\mathcal{T}$ is homeomorphic to $\mathrm{Spc}(\mathcal{A})$.
\end{lem}

 \section{\bf group prime and support}

In this section, we work with $\mathcal{U}$ such that $\mathrm{A}(\mathcal{U})$ is locally noetherian. Then the subcategory $\mathrm{A}(\mathcal{U})^{c}$ of finitely generated objects is a noetherian tensor abelian category (under the hypothesis that $\mathcal{U}$  is widely closed and the unit $\mathrm{1}$ is finitely generated). There is a class of prime Serre ideals in $\mathrm{A}(\mathcal{U})^{c}$, called group primes. We prove that when $\mathcal{U}$ is essentially finite, then every prime in $\mathrm{A}(\mathcal{U})^{c}$ is a group prime. An important step is to characterize when $X\in \mathrm{Serre}_{\otimes}\langle Y\rangle$ for $X,Y\in \mathrm{A}(\mathcal{U})^{c}$. This result will later be extended to more general families $\mathcal{U}$, which is crucial for determining the spectrum of finitely generated VI-modules.

\begin{hyp} We assume that any~$\mathcal{U}$ is widely closed and the unit $\mathrm{1}\in \mathrm{A}(\mathcal{U})$  is finitely generated.
\end{hyp}
\begin{df} We say $\mathcal{U}$ is noetherian if $\mathrm{A}(\mathcal{U})$ is locally noetherian, i.e., every subobject of $e_{G}$ is finitely generated for each $G\in \mathcal{U}$. We denote by $\mathrm{A}(\mathcal{U})^{c}$ the  abelian subcategory of finitely generated objects.
\end{df}

\begin{ex} Each of the following family is noetherian, see \cite[Lemma 11.9, Theorem 13.4]{PS22}.
\begin{enumerate}
\item $\mathcal{U}$ is essentially finite.

\item $\mathcal{U}$ is the subcategory of free $\mathbb{Z}/p^{n}$-modules.

\item $\mathcal{U}$ is the subcategory of cyclic $p$-groups.

\item $\mathcal{U}$ is the subcategory of finite abelian $p$-groups.
\end{enumerate}
\end{ex}

\begin{prop}\label{prop4.4} If $\mathcal{U}$ is noetherian, then $\mathrm{A}(\mathcal{U})^{c}$ is a tensor  abelian category.
\end{prop}
\begin{proof}Since $\mathrm{1}\in \mathrm{A}(\mathcal{U})^{c}$ and $e_{G}$ is projective for each $G\in \mathcal{U}$, then by \cite[Proposition 8.7]{PS22}, the subcategory of finitely generated objects is closed under tensor product.
\end{proof}

\begin{df} Let $\mathcal{U}$ be noetherian and $G\in \mathcal{U}$. Define the subcategory $P_{G}\subset \mathrm{A}(\mathcal{U})^{c}$ by $$P_{G}:=\{X\in \mathrm{A}(\mathcal{U})^{c}|~X(G)=0\}.$$ We call such a subcategory a group prime.
\end{df}
\begin{lem}\label{lem4.6}Let $\mathcal{U}$ be noetherian. Then the (index) support satisfies the following properties:
\begin{enumerate}
\item $\mathrm{supp}(0)=\emptyset$ and $\mathrm{supp}(\mathrm{1})=\pi_{0}(\mathcal{U})$;

\item $\mathrm{supp}(X\oplus Y)=\mathrm{supp}(X)\cup \mathrm{supp}(Y)$ for $X,Y\in \mathrm{A}(\mathcal{U})^{c}$;

\item $\mathrm{supp}(Y)=\mathrm{supp}(X)\cup \mathrm{supp}(Y)$ for any short exact sequence $X\rightarrowtail Y\twoheadrightarrow Z$ in $\mathrm{A}(\mathcal{U})^{c}$;

\item $\mathrm{supp}(X\otimes Y)=\mathrm{supp}(X)\cap \mathrm{supp}(Y)$ for $X,Y\in \mathrm{A}(\mathcal{U})^{c}$;

\item ~$\mathrm{supp}(X)=\bigcup_{Z\in \mathrm{Serre}_{\otimes}\langle X\rangle}\mathrm{supp}(Z)$ for $X\in \mathrm{A}(\mathcal{U})^{c}$.

\end{enumerate}
In particular, $(\pi_{0}(\mathcal{U}), \mathrm{supp})$ is a support datum on $\mathrm{A}(\mathcal{U})^{c}$, which induces a continuous injective map $\pi_{0}(\mathcal{U})\rightarrow \mathrm{Spc}(\mathcal{A})$, $[G]\mapsto P_{G}$.
\end{lem}
\begin{proof}The proofs of part (1)-(5) are standard.

 By \cite[Theorem 5.3]{BKS}, there exists a continuous map $f:\pi_{0}(\mathcal{U})\rightarrow \mathrm{Spc}(\mathcal{A})$ such that $\mathrm{supp}(X)=f^{-1}(Z(\{X\}))$ for each $X\in \mathrm{A}(\mathcal{U})^{c}$. By construction, $X(G)=0\Leftrightarrow G\notin \mathrm{supp}(X)\Leftrightarrow G\notin f^{-1}(Z(\{X\}))\Leftrightarrow f(G)\notin Z(\{X\})\Leftrightarrow X\in f(G)$ and so $f([G])=P_{G}$.

  To prove injectivity, suppose $P_{G}=P_{H}$. Then $e_{G}\notin P_{G}$ implies that $e_{G}\notin P_{H}$ and so $e_{G}(H)\neq 0$. Thus there is an epimorphism $H\twoheadrightarrow G$. Similarly there is an epimorphism $G\twoheadrightarrow H$. Thus $G\cong H$, as claimed.
\end{proof}

\begin{lem}\label{lem4.7} Let $\mathcal{U}$ be noetherian and $G\in \mathcal{U}$. Suppose $V\in \mathrm{A}(\{G\})^{c}$ is nonzero, then:
\begin{enumerate}
\item $\mathrm{Serre}_{\otimes}\langle e_{G,V}\rangle=\mathrm{Serre}_{\otimes}\langle e_{G}\rangle\subseteq \mathrm{A}(\mathcal{U})^{c}$.

\item $\mathrm{Serre}_{\otimes}\langle \chi_{G,V}\rangle=\mathrm{Serre}_{\otimes}\langle \chi_{G,k}\rangle\subseteq \mathrm{A}(\mathcal{U})^{c}$.

\end{enumerate}
\end{lem}
\begin{proof} The proof is the same as in \cite[Lemma 2.9, Proposition 2.11]{BBP+25b}.
\end{proof}

\begin{lem}\label{lem4.8} Let $\mathcal{U}$ be noetherian and $G\in \mathcal{U}$. Suppose $V\in \mathrm{A}(\{G\})^{c}$ is nonzero, there is a short exact sequence in $\mathrm{A}(\mathcal{U})^{c}$
$$F\rightarrowtail e_{G,V}\twoheadrightarrow \chi_{G,V}$$
where the right map is the surjection adjoint to the identity map of $V$ and $F\in \mathrm{Serre}_{\otimes}\langle e_{H}~| H\in \uparrow (\{G\})\setminus\{G\} \rangle$.
\end{lem}
\begin{proof}Let $F$ be the kernel of the surjection and $i:\uparrow (\{G\})\setminus\{G\} \rightarrow \mathcal{U}$  the inclusion of the upwards-closed subcategory . Then $F$ is finitely generated and the counit map $i_{!}i^{*}F\rightarrow F$ is an isomorphism since $i_{!}$ is extension by zero and $F(H)\cong 0$ for all $H\notin \uparrow (\{G\})\setminus\{G\}$. For $i^{*}F\in \mathrm{A}(\uparrow (\{G\})\setminus\{G\})$, we have $$F\cong i_{!}i^{*}F\in \mathrm{Serre}_{\otimes}\langle e_{H}|~H\in \uparrow (\{G\})\setminus\{G\}\rangle,$$ as claimed.
\end{proof}

\begin{prop}\label{prop4.9} Let $\mathcal{U}$ be essentially finite and $X,Y\in \mathrm{A}(\mathcal{U})^{c}$. Then $$\mathrm{Serre}_{\otimes}\langle X\rangle=\mathrm{Serre}_{\otimes}\langle \chi_{G,k}|~G\in \mathrm{supp}(X)\rangle.$$
In particular, we have $\mathrm{supp}(X)\subseteq \mathrm{supp}(Y)$ if and only if $X\in \mathrm{Serre}_{\otimes}\langle Y\rangle$.
\end{prop}
\begin{proof}For $\supseteq$ part, suppose $G\in \mathrm{supp}(X)$. Then there is a map $e_{G,X(G)}\rightarrow X$ induced by the identity map of $X(G)$. Tensoring with $\chi_{G,k}$, we obtain isomorphisms in $\mathrm{A}(\mathcal{U})^{c}$ $$\chi_{G,X(G)}\cong \chi_{G,k}\otimes e_{G,X(G)}\cong \chi_{G,k}\otimes X.$$
Thus by Lemma \ref{lem4.7}, we have $$\chi_{G,k}\in \mathrm{Serre}_{\otimes}\langle \chi_{G,X(G)}\rangle\subseteq \mathrm{Serre}_{\otimes}\langle X\rangle.$$
For the other containment, we prove by induction on $n=|\mathrm{supp}(X)|$. If $n=0$, there is nothing to say.  Now suppose $n\geq 1$ and the claim holds for all $Z\in \mathrm{A}(\mathcal{U})^{c}$ with $|\mathrm{supp}(X)|< n$. Let $\pi_{0}(\mathcal{U})=\{[G_{1}],...,[G_{n}]\}$ and choose $G_{n}$ maximal among the $G_{i}$s with respect to epimorphism. Then there is a  diagram $$\xymatrix@C=0.5cm{
  F \ar[rr]\ar[d] && e_{G_{n},X(G_{n})} \ar[rr]\ar[d]^{g_{n}} && \chi_{G_{n},X(G_{n})} \ar@{-->}[d]^{f_{n}} \\
  0 \ar[rr] && X \ar[rr]^{\cong} && X  }$$
where $F\in \mathrm{Serre}_{\otimes}\langle e_{H}~| H\in \uparrow (\{G_{n}\})\setminus\{G_{n}\} \rangle$ by Lemma \ref{lem4.8} and $g_{n}$ is the map induced by the identity map of $X(G_{n})$. By construction we have $\mathrm{supp}(X)\cap \mathrm{supp}(F)=\emptyset$, thus the left square is commutative and so there exists $f_{n}: \chi_{G_{n},X(G_{n})}\rightarrow X$ making the right square commutative. Note that $f_{n}$ is a monomorphism, define $W:=\mathrm{Coker}(f_{n})$. Then $\mathrm{supp}(W)=\mathrm{supp}(X)\setminus \{G_{n}\}$. By induction hypothesis and Lemma \ref{lem4.7}, we have $$X\in \mathrm{Serre}_{\otimes}\langle \chi_{G_{n},X(G_{n})},W\rangle\subseteq \mathrm{Serre}_{\otimes}\langle \chi_{G_{i},k}|~1\leq i\leq n\rangle,$$ as claimed.
\end{proof}

\begin{thm} Let $\mathcal{U}$ be essentially finite. Then $\mathrm{Spc}(\mathrm{A}(\mathcal{U})^{c})\cong \pi_{0}(\mathcal{U})$.
\end{thm}
\begin{proof} By Proposition \ref{prop4.9}, each Serre ideal in $\mathrm{A}(\mathcal{U})^{c}$ is radical and determined by its support. Thus $(\pi_{0}(\mathcal{U}), \mathrm{supp})$ is a  classifying support datum on $\mathrm{A}(\mathcal{U})^{c}$. Then the result follows from Lemma \ref{lem 3.8}.
\end{proof}

\begin{lem}\label{lem4.12} Let $i:\mathcal{U}\rightarrow \mathcal{V}$ be the inclusion of a down-closed subcategory  and $j:\mathcal{V}\setminus~\mathcal{U}\rightarrow \mathcal{V}$ be the complement. Suppose $\mathcal{V}$ is noetherian, we have:
\begin{enumerate}
\item $j_{!}j^{*}X\rightarrow X\rightarrow i_{*}i^{*}X$ is short exact for any $X\in \mathrm{A}(\mathcal{V})^{c}$ and thus $\mathrm{A}(\mathcal{V})^{c}/\mathrm{A}(\mathcal{V}\setminus \mathcal{U})^{c}\simeq_{\otimes} \mathrm{A}(\mathcal{U})^{c} $;

\item if $X\in \mathrm{A}(\mathcal{V})^{c}$ and $ \mathrm{supp}(X)\subseteq \mathcal{V}\setminus~\mathcal{U}$, then $j_{!}j^{*}X\cong X$;

\item $X\in \mathrm{Serre}_{\otimes}\langle e_{G}|~G\in \uparrow(\mathrm{supp}(X))\rangle$ for any $X\in \mathrm{A}(\mathcal{V})^{c}$.

\end{enumerate}
\end{lem}
\begin{proof}For part (1), the result follows from the fact that both $j_{!}$ and $i_{*}$ are extension by zero.
For part (2), by the assumption we have $i^{*}X=0$ and so $j_{!}j^{*}X\cong X$. The claim (3) follows from (2) applied to $\mathcal{V}\setminus \mathcal{U}= \uparrow(\mathrm{supp}(X))$ and the fact that $j_{!}$ preserves tensors.
\end{proof}

\begin{lem}\label{lem4.13} Let $i:\mathcal{U}\rightarrow \mathcal{V}$ be the inclusion of a down-closed subcategory  and $j:\mathcal{V}\setminus~\mathcal{U}\rightarrow \mathcal{V}$ be the complement. Suppose $\mathcal{V}$ is noetherian and $X\in \mathrm{A}(\mathcal{U})^{c}$, then we have: $$i_{!}\mathrm{Serre}_{\otimes}\langle X\rangle\subseteq \mathrm{Serre}_{\otimes}\langle \{i_{!}X\}\cup \{e_{G}|~G\in \mathcal{V}\setminus~\mathcal{U}\}\rangle.$$
\end{lem}
\begin{proof} Firstly we claim that $i_{!}(X\otimes Y)\in \mathrm{Serre}_{\otimes}\langle \{i_{!}X\}\cup \{e_{G}|~G\in \mathcal{V}\setminus~\mathcal{U}\}\rangle$ for any $Y\in \mathrm{A}(\mathcal{U})^{c}$.  Consider the following diagram: $$\xymatrix{
                & i_{!}(X\otimes Y) \ar@{->>}[dr] \ar[rr]^{h} & &i_{!}(X)\otimes i_{!}(Y) \ar@{->>}[rr] & & \mathrm{Coker}(h)         \\
 \mathrm{Ker}(h) \ar@{>->}[ur] & &     \mathrm{Im}(h)\ar@{>->}[ur]       }
$$
where the map $h$ is induced by the oplax monoidal structure on $i_{!}$. Since $i^{*}$ is strong monoidal and exact, we have $i^{*}\mathrm{Coker}(h)=0$ and so $\mathrm{Coker}(h)\in \mathrm{A}(\mathcal{U})^{c}$ by Lemma \ref{lem4.12}. Thus $$\mathrm{Coker}(h)\cong j_{!}j^{*}\mathrm{Coker}(h)\in \mathrm{Serre}_{\otimes}\langle \{e_{G}|~G\in \mathcal{V}\setminus~\mathcal{U}\}\rangle$$ by the same lemma again. Similarly $i^{*}\mathrm{Ker}(h)=0$ and so $$\mathrm{Ker}(h)\cong j_{!}j^{*}\mathrm{Ker}(h)\in \mathrm{Serre}_{\otimes}\langle \{e_{G}|~G\in \mathcal{V}\setminus~\mathcal{U}\}\rangle.$$
Thus \begin{align*}
  i_{!}(X\otimes Y) &\in \mathrm{Serre}_{\otimes}\langle \mathrm{Ker}(h), \mathrm{Im}(h), \mathrm{Coker}(h) \rangle\\
  &\subseteq \mathrm{Serre}_{\otimes}\langle \mathrm{Ker}(h), i_{!}(X)\otimes i_{!}(Y) \rangle\\
  &\subseteq \mathrm{Serre}_{\otimes}\langle  \{i_{!}X\}\cup \{e_{G}|~G\in \mathcal{V}\setminus~\mathcal{U}\} \rangle,
\end{align*}
as claimed.

Now suppose $Y\rightarrowtail X\twoheadrightarrow Z$ is short exact. Then $$ i_{!}(Z)\in \mathrm{Serre}_{\otimes}\langle  i_{!}X \rangle\subseteq \mathrm{Serre}_{\otimes}\langle  \{i_{!}X\}\cup \{e_{G}|~G\in \mathcal{V}\setminus~\mathcal{U}\} \rangle$$ since $i_{!}$ is right exact. Apply $i_{!}$ to the injection $f:Y\rightarrowtail X$, we have $i^{*}\mathrm{Ker}(i_{!}f)=0$ and so $\mathrm{Ker}(i_{!}f)\in \mathrm{Serre}_{\otimes}\langle \{e_{G}|~G\in \mathcal{V}\setminus~\mathcal{U}\rangle$ as above. Then $$i_{!}(Y)\in \mathrm{Serre}_{\otimes}\langle \mathrm{Ker}(i_{!}f), \mathrm{Im}(i_{!}f)\rangle \subseteq \mathrm{Serre}_{\otimes}\langle  \{i_{!}X\}\cup \{e_{G}|~G\in \mathcal{V}\setminus~\mathcal{U}\} \rangle.$$

At last, for any short exact sequence $Y\rightarrowtail W\twoheadrightarrow Z$ in $\mathrm{A}(\mathcal{U})^{c}$ with $i_{!}(Y), i_{!}(Z)\in  \mathrm{Serre}_{\otimes}\langle  \{i_{!}X\}\cup \{e_{G}|~G\in \mathcal{V}\setminus~\mathcal{U}\} \rangle$, then \begin{align*}
  i_{!}(W) &\in \mathrm{Serre}_{\otimes}\langle  \mathrm{Im}(g), i_{!}(Z) \rangle\\
  &\subseteq \mathrm{Serre}_{\otimes}\langle i_{!}(Y), i_{!}(Z) \rangle\\
  &\subseteq \mathrm{Serre}_{\otimes}\langle  \{i_{!}X\}\cup \{e_{G}|~G\in \mathcal{V}\setminus~\mathcal{U}\} \rangle
\end{align*} where $g:Y\rightarrow W$ is the monomorphism. This concludes the proof.
\end{proof}

\begin{df} \label{df 4.13} Let $\mathcal{V}$ be noetherian and $X\in \mathrm{A}(\mathcal{V})^{c}$. Define a Serre ideal $\mathrm{Serre}_{\otimes,n}^{+}\langle X\rangle\subseteq \mathrm{A}(\mathcal{V})^{c}$ as $$\mathrm{Serre}_{\otimes,n}^{+}\langle X\rangle:=\mathrm{Serre}_{\otimes}\langle \{X\}\cup \{e_{G}| G\in (\uparrow(\mathrm{supp}(X))_{>n}\}\rangle$$ and set $$\mathrm{Serre}_{\otimes}^{+}\langle X\rangle:=\bigcap_{n>0}\mathrm{Serre}_{\otimes,n}^{+}\langle X\rangle.$$
\end{df}

Compare the following result to \cite[Proposition 5.8]{BBP+25b}.
\begin{thm}\label{thm 4.15} Let $\mathcal{V}$ be noetherian and $X, Y\in \mathrm{A}(\mathcal{V})^{c}$. If $\mathrm{supp}(X)\subseteq \mathrm{supp}(Y)$, then $X\in \mathrm{Serre}_{\otimes}^{+}\langle Y\rangle$.
\end{thm}
\begin{proof} We first prove the claim in the special case $\uparrow(\mathrm{supp}(Y))=\mathcal{V}$. Since $X, Y\in \mathrm{A}(\mathcal{V})^{c}$, there exists some finite set $S\subseteq \mathcal{V}$ such that $X,Y\in \mathrm{Serre}_{\otimes}\langle e_{G}|~G\in S\rangle$. Let $n\geq \mathrm{max}\{|G|~|G\in S\}$ and write $i:\mathcal{V}_{\leq n}\rightarrow \mathcal{V}$ for the inclusion of the downwards-closed subcategory. Consider the following diagram:$$\xymatrix{
                & i_{!}i^{*}X \ar@{->>}[dr] \ar[rr]^{g} & &X  \ar@{->>}[rr] & & \mathrm{Coker}(g)        \\
 & &      \mathrm{Im}(g)\ar@{>->}[ur]       }
$$where g is the counit map on $X$, we have $\mathrm{supp}(\mathrm{Coker}(g))\subset \mathcal{V}_{>n}$. Similarly we have $\mathrm{supp}(\mathrm{Ker}(f))\subset \mathcal{V}_{>n}$ where f is the counit map on $Y$. Since $\mathrm{supp}(i^{*}X)\subseteq \mathrm{supp}(i^{*}X)$, we have $i^{*}X\in \mathrm{Serre}_{\otimes}\langle i^{*}Y\rangle$ by Proposition \ref{prop4.9}. Then \begin{align*}
  X &\in \mathrm{Serre}_{\otimes}\langle  \mathrm{Im}(g), \mathrm{Coker}(g) \rangle\\
  &\subseteq \mathrm{Serre}_{\otimes}\langle  i_{!}i^{*}(X), \{e_{G}|~G\in \mathcal{V}_{>n}\} \rangle \quad\quad\quad\quad\quad \mathrm{by~Lemma~\ref{lem4.12}}\\
  &\subseteq \mathrm{Serre}_{\otimes}\langle  i_{!}i^{*}(Y), \{e_{G}|~G\in \mathcal{V}_{>n}\} \rangle \quad\quad\quad\quad\quad \mathrm{by~Lemma~\ref{lem4.13}}\\
  &\subseteq \mathrm{Serre}_{\otimes}\langle  \mathrm{Im}(f),\mathrm{Ker}(f), \{e_{G}|~G\in \mathcal{V}_{>n}\} \rangle \\
  &\subseteq \mathrm{Serre}_{\otimes}\langle  Y, \{e_{G}|~G\in \mathcal{V}_{>n}\} \rangle \\
  &= \mathrm{Serre}_{\otimes}\langle  Y, \{e_{G}|~G\in \uparrow(\mathrm{supp}(Y))_{>n}\} \rangle .
\end{align*}
Since $n$ could be taken arbitrarily large, we have that $X\in \mathrm{Serre}_{\otimes}^{+}\langle Y\rangle$, as claimed.

Now write $j:\uparrow \mathrm{supp}(Y)\rightarrow \mathcal{V}$ for the inclusion of the upwards-closed subcategory. Then $j_{!}j^{*}\rightarrow \mathrm{id}$ is an isomorphism on X and Y and so $$j_{!}\mathrm{Serre}_{\otimes,n}^{+}\langle j^{*}X\rangle\subseteq \mathrm{Serre}_{\otimes,n}^{+}\langle j_{!}j^{*}X\rangle=\mathrm{Serre}_{\otimes,n}^{+}\langle X\rangle$$ holds for any $n\geq 0$. By the above claim and the fact that $\mathrm{supp}(j^{*}X)\subseteq \mathrm{supp}(j^{*}Y)$, we have $$X\cong j_{!}j^{*}(X)\in j_{!}\mathrm{Serre}_{\otimes,n}^{+}\langle j^{*}Y\rangle\subseteq \mathrm{Serre}_{\otimes,n}^{+}\langle Y\rangle.$$
The result follows by taking the intersection over varying $n$.
\end{proof}

 \section{\bf Prime spectrum of VI-modules}
In this section, we determine the spectrum of prime Serre ideals for global representations over a noetherian and $\mathbb{N}$-stable family $\mathcal{U}$. Primary examples include the family of elementary abelian p-groups and the family of cyclic p-groups.

\begin{df}\label{df 5.1}
A subcategory $\mathcal{U} \subseteq \mathcal{G}$ is called $\mathbb{N}$-stable if it satisfies the following:
\begin{enumerate}
    \item  The set $\pi_{0}(\mathcal{U})$ of isomorphism classes of objects in $\mathcal{U}$ is in bijection with $\mathbb{N}$; we write these classes as $\{[G_n]\}_{n \in \mathbb{N}}$.
    \item $\pi_{0}(\mathcal{U})$ is totally ordered by the existence of epimorphisms:
          for any distinct $m, n \in \mathbb{N}$, exactly one of the following holds:
          \begin{itemize}
              \item there exists an epimorphism $G_m \twoheadrightarrow G_n$,
              \item there exists an epimorphism $G_n \twoheadrightarrow G_m$.
          \end{itemize}
          Moreover, the direction is consistent with the natural order on indices: if $m \le n$, then there exists an epimorphism $G_n \twoheadrightarrow G_m$.
    \item Every finitely generated object $X \in \mathrm{A}(\mathcal{U})$ is eventually torsion-free: there exists $r \in \mathbb{N}$ such that for every epimorphism $\alpha: G_m \twoheadrightarrow G_n$ with $n > r$, the induced map $\alpha^{*}: X(G_n) \rightarrow X(G_m)$ is injective.
\end{enumerate}
\end{df}

\begin{ex}Let $\mathcal{U}$ be the family of elementary abelian p-groups or the family of cyclic p-groups. Then $\mathcal{U}$ is $\mathbb{N}$-stable and noetherian by \cite[Theorem B]{PS22}.
\end{ex}

\begin{nota}For simplicity, we write $e_{n,V}$ for $e_{G_{n},V}$,~$X(n)$ for $X(G_{n})$,~$P_{n}$ for $P_{G_{n}}$ and so on.
\end{nota}

\begin{df}Let $\mathcal{U}$ be $\mathbb{N}$-stable and noetherian. Define a subcategory $P_{\infty}\subseteq \mathrm{A}(\mathcal{U})^{c}$ as $$P_{\infty}:=\{X\in \mathrm{A}(\mathcal{U})^{c}|~\mathrm{supp}(X)\subseteq\mathbb{N}~is~finite\}.$$
\end{df}

\begin{lem}\label{lem 5.5}Let $\mathcal{U}$ be $\mathbb{N}$-stable and noetherian. Then  $$P_{\infty}=\mathrm{Serre}_{\otimes}\langle\chi_{i,k}|~i\in \mathbb{N}\rangle$$
is a minimal prime Serre ideal.
\end{lem}
\begin{proof}By Lemma \ref{lem4.6}, $P_{\infty}$ is a Serre ideal. Suppose $X,Y\in \mathrm{A}(\mathcal{U})^{c}$ with $X\otimes Y\in P_{\infty}$. If $\mathrm{supp}(X)$ and $\mathrm{supp}(Y)$ are not finite, then by the condition that $X$ and $Y$ are eventually torsion free, there exists $r\in \mathbb{N}$ such that $\mathrm{dim}_{k}(X(n))>0, \mathrm{dim}_{k}(Y(n))>0$ for all $n\geq r$. This would imply that $(X\otimes Y)(n)\neq 0$ for sufficiently large $n$, contradicting $X \otimes Y \in P_{\infty}$. Thus $P_{\infty}$ is prime.

The containment $\supseteq$ is obvious. For the other direction, suppose $X$ is a finitely generated object supported on $[m,n]$. Then we have the following diagram:$$\xymatrix{
  \Gamma_{m}X:= \ar@{>->}[d]  & \cdots \ar[d] \ar[r] & 0 \ar[d] \ar[r] & 0 \ar[d] \ar[r] & X(m+1) \ar[d]^{\simeq} \ar[r] & \cdots \ar[d]^{\simeq}\ar[r]& X(n) \ar[d]^{\simeq}\ar[r] & 0\ar[d]\ar[r]&\cdots\ar[d]\\
  X:= \ar@{->>}[d]  & \cdots \ar[d] \ar[r] & 0 \ar[d] \ar[r] & X(m) \ar[d]^{\simeq} \ar[r] & X(m+1) \ar[d] \ar[r] & \cdots \ar[d]\ar[r]& X(n)\ar[d]\ar[r] & 0\ar[d]\ar[r]&\cdots\ar[d] \\
  \chi_{m,X(m)}:=  & \cdots \ar[r] & 0 \ar[r] & X(m) \ar[r] & 0 \ar[r] & \cdots \ar[r] & 0 \ar[r] & 0\ar[r]&\cdots}$$
  which implies that $X\in \mathrm{Serre}_{\otimes}\langle\Gamma_{m}X, \chi_{m,X(m)}\rangle\subseteq \mathrm{Serre}_{\otimes}\langle\Gamma_{m}X, \chi_{m,k}\rangle.$ Inductively we have $X\in \mathrm{Serre}_{\otimes}\langle\chi_{i,k}|~m\leq i\leq n\rangle\subseteq \mathrm{Serre}_{\otimes}\langle\chi_{i,k}|~i\in \mathbb{N}\rangle.$

 Finally, suppose $P$ is any prime Serre ideal in $\mathrm{A}(\mathcal{U})^{c}$ such that $P\subseteq P_{\infty}$. We claim that $P=P_{\infty}$. Denote the following object by $\gamma_{i}$: $$\xymatrix@C=0.5cm{
  k \ar[r]^{=} & \cdots \ar[r]^{=} &k  \ar[r] & 0 \ar[r] & k \ar[r]^{=} &\cdots  }$$
  where 0 is in degree i. It is easy to show $\gamma_{i}$ is finitely generated. Then  $\chi_{i,k}\in P$ for each $i\in \mathbb{N}$ follows from the fact that $\gamma_{i}\notin  P_{\infty}$ and $\gamma_{i}\otimes \chi_{i,k}=0\in P$. Thus $P_{\infty}=\mathrm{Serre}_{\otimes}\langle\chi_{i,k}|~i\in \mathbb{N}\rangle\subseteq P$ and so $P=P_{\infty}$, as claimed.
\end{proof}

\begin{lem}\label{lem 5.6}Let $\mathcal{U}$ be  noetherian and $\mathbb{N}$-stable. Suppose $X,Y\in \mathrm{A}(\mathcal{U})^{c}$ with $\mathrm{supp}(X)\subseteq \mathrm{supp}(Y)$, then $X\in \mathrm{Serre}_{\otimes}\langle Y\rangle$.
\end{lem}
\begin{proof} Note that $\mathrm{supp}(Y)$ is either finite or cofinite. In the case $\mathrm{supp}(Y)$ is  finite, we have $\mathrm{Serre}_{\otimes}\langle Y\rangle=\mathrm{Serre}_{\otimes}\langle \chi_{i,k}|~i\in \mathrm{supp}(Y)\rangle$, as the same way in the proof of Lemma \ref{lem 5.5}. Similarly $\mathrm{Serre}_{\otimes}\langle X\rangle=\mathrm{Serre}_{\otimes}\langle \chi_{i,k}|~i\in \mathrm{supp}(X)\rangle$ and so $X\in \mathrm{Serre}_{\otimes}\langle Y\rangle$.

Now suppose $\mathrm{supp}(Y)$ is cofinite. Then there exist some $n_{1}\in \mathbb{N}$ and some torsion-free element $y\in Y(n_{1})$ such that $\alpha^{*}(y)\neq 0$ for any epimorphism $\alpha$. This would yield a monomorphism $e_{n}\rightarrowtail e_{n}\otimes Y$ for any $n\geq n_{1}$, as shown in the proof of \cite[Lemma 7.10]{BBP+25a}. Thus we have $e_{n}\in \mathrm{Serre}_{\otimes}\langle Y\rangle$ when $n\geq n_{1}$. Therefore $\mathrm{Serre}_{\otimes}^{+}\langle Y\rangle=\mathrm{Serre}_{\otimes}\langle Y\rangle$. Our claim then follows from Theorem \ref{thm 4.15}.
\end{proof}

\begin{thm}Let $\mathcal{U}$ be  noetherian and $\mathbb{N}$-stable. Then the only prime Serre ideals in $\mathrm{A}(\mathcal{U})^{c}$ are $P_{n}$ for $n\in \mathbb{N}$ and $P_{\infty}$. The spectrum of prime Serre ideals of  $\mathrm{A}(\mathcal{U})^{c}$ is homeomorphic to $\mathbb{N}^{*}$, the one-point compactification of $\mathbb{N}$.
\end{thm}
\begin{proof}Firstly we have that $P_{n}$ is maximal for $n\in \mathbb{N}$. Indeed, it is easy to verify that $\mathrm{A}(\mathcal{U})^{c}/~P_{n}\simeq \mathrm{A}(\{n\})\simeq Vec_{k}$. Then by Proposition \ref{prop3.7}, the only prime Serre ideal containing $P_{n}$ is itself and so $P_{n}$ is maximal for $n\in \mathbb{N}$.

Now suppose $P$ is some prime Serre ideal in $\mathrm{A}(\mathcal{U})^{c}$ such that $P\neq P_{\infty}$. If $P\subseteq P_{\infty}$ then we have $P= P_{\infty}$ by Lemma \ref{lem 5.5}, contradicting to our assumption. Thus $P\varsubsetneq P_{\infty}$. Let $X\in P, X\notin P_{\infty}$. Then $\mathrm{supp}(X)$ is cofinite. Set $S:=\{n\in \mathbb{N}|~X(n)=0\}$. Note that $S$ is not empty, otherwise $\mathrm{supp}(X)=\mathrm{supp}(\mathrm{1})=\mathbb{N}$, which implies that $\mathrm{1}\in \mathrm{Serre}_{\otimes}\langle X\rangle\subseteq P$ by Lemma \ref{lem 5.6}, a contradiction. Now we have $\mathrm{supp}(\otimes(\gamma_{i})_{i\in S})=\mathrm{supp}(X)$ where $\gamma_{i}$ is  as in the proof of Lemma \ref{lem 5.5}. Thus $\otimes(\gamma_{i})_{i\in S}\in \mathrm{Serre}_{\otimes}\langle X\rangle\subseteq P$ by Lemma \ref{lem 5.6} again. Therefore we have $\gamma_{i}\in P$ for some $i\in S$. For any $Y\in P_{i}$, we have $\mathrm{supp}(Y)\subseteq \mathrm{supp}(\gamma_{i})$ and so $Y\in \mathrm{Serre}_{\otimes}\langle \gamma_{i}\rangle\subseteq P$. Thus $P_{i}\subseteq P$, which implies that $P=P_{i}$ since $P_{i}$ is maximal. Then we conclude our first claim.

For the second claim, note that  $\mathrm{supp}(X)$ is  either finite or cofinite for each $X\in \mathrm{A}(\mathcal{U})^{c}$. Then $Z(\{X\})$ is either of the form $S\subseteq \mathbb{N}$ where $S$ is finite or of the form $P_{\infty}\bigcup C$ for some cofinite subset $C\subseteq \mathbb{N}$.
\end{proof}

\begin{cor}Let $\mathcal{U}$ be the family of elementary abelian p-groups or the family of cyclic p-groups. Then $\mathrm{Spc}(\mathrm{A}(\mathcal{U})^{c})\cong \mathbb{N}^{*}$.
\end{cor}

Note that when  $\mathcal{U}$ is the family of elementary abelian p-groups, there is an equivalence between $\mathrm{A}(\mathcal{U})^{c}$ and the category of finitely generated VI-modules (over $k$).
\begin{cor} The spectrum of prime Serre ideals of finitely generated VI-modules is homeomorphic to $\mathbb{N}^{*}$.
\end{cor}

\begin{rem}Our method could be applied to $\mathrm{FI}_{k}^{c}$, the category  of finitely generated FI-modules, since each notion appeared in this article has a counterpart in $\mathrm{FI}_{k}^{c}$; see \cite{CEFN14}. Thus we have $\mathrm{Spc}(\mathrm{FI}_{k}^{c})\cong \mathbb{N}^{*}$.
\end{rem}

\begin{rem}Let $\mathrm{D}(\mathcal{U})^{c}$ denote the category of compact objects in the derived category of $\mathrm{A}(\mathcal{U})$. It is shown in \cite[Theorem 6.13, Example 12.10]{BBP+25b} that the Balmer spectrum of $\mathrm{D}(\mathcal{U})^{c}$ is homeomorphic to the Hochster dual of $\mathrm{Spec}(\mathbb{Z})$ when $\mathcal{U}$ is the family of elementary abelian p-groups, while $\mathrm{Spc}(\mathrm{D}(\mathcal{U})^{c})\cong \mathbb{N}^{*}$ when $\mathcal{U}$ is the family of cyclic p-groups. Thus it is natural to ask the following question.
\end{rem}

\begin{Quest}For a noetherian family $\mathcal{U}$, when we have $\mathrm{Spc}(\mathrm{A}(\mathcal{U})^{c})\cong \mathrm{Spc}(\mathrm{D}(\mathcal{U})^{c})$?
\end{Quest}

\renewcommand\refname
\vspace{4mm}

\textbf{Peng Xu}\\
School of Mathematics, Nanjing University,
Nanjing 210093, P. R. China;\\
E-mail: \textsf{602023210017@smail.nju.edu.cn}\\[1mm]

\end{document}